\documentclass[12pt,reqno]{amsart}
\usepackage{etoolbox}

   \makeatletter

 \patchcmd{\@setaddresses}{\scshape\ignorespaces}{\ignorespaces}{}{} 

\appto\maketitle{%
\let\@makefnmark\relax  \let\@thefnmark\relax
\ifx\@empty\addresses\else\@footnotetext{%
  \vskip-\bigskipamount\@setaddresses}
  }
\def\enddoc@text{}
\makeatother

\makeatletter
\patchcmd\maketitle
  {\uppercasenonmath\shorttitle}
  {}
  {}{}
\patchcmd\maketitle
  {\@nx\MakeUppercase{\the\toks@}}
  {\the\toks@}
  {}
  {}{}
\patchcmd\@settitle{\uppercasenonmath\@title}{\Large}{}{}
\patchcmd\@setauthors
  {\MakeUppercase{\authors}}
  {\authors}
  {}{}
\makeatother
\usepackage{amsmath,amssymb,amsthm}
\usepackage{color}
\usepackage{url}
\usepackage{tikz-cd}
\usepackage[utf8]{inputenc}
\usepackage[T1]{fontenc}
\usepackage{geometry}
\newtheorem{theorem}{Theorem}[section]
\newtheorem{definition}{Definition}[section]

\newtheorem{corollary}{Corollary}[section]
\newtheorem{proposition}{Proposition}[section]
\newtheorem{lemma}{Lemma}[section]
\newtheorem{remark}{Remark}[section]
\newtheorem{example}{Example}[section]
\numberwithin{equation}{section}
\usepackage[colorlinks=true]{hyperref}  
\hypersetup{urlcolor=blue, citecolor=red, linkcolor=blue}

\usepackage[capitalise,noabbrev,nameinlink]{cleveref}      
\begin{document}
\address{$^{[1]}$ University of Sfax, Tunisia.}
\email{\url{kais.feki@hotmail.com}}

\address{$^{[2]}$ Mathematics Department, College of Science, Jouf University, P.O. Box 2014, Sakaka, Saudi Arabia.}
\email{\url{sidahmed@ju.edu.sa}}



\date{\today}
\author[Kais Feki and Sid Ahmed Ould Ahmed Mahmoud] {\Large{Kais Feki}$^{1}$ and \Large{Sid Ahmed Ould Ahmed Mahmoud}$^{2}$ }
\title[Davis-Wielandt shells of semi-Hilbertian space operators and its applications]{Davis-Wielandt shells of semi-Hilbertian space operators and its applications}

\maketitle

\begin{abstract}
In this paper we generalize the concept of Davis-Wielandt shell of operators on a Hilbert space when a semi-inner product induced by a positive operator $A$ is considered. Moreover, we investigate the parallelism of $A$-bounded operators with respect to the seminorm and the numerical radius induced by $A$. Mainly, we characterize $A$-normaloid operators in terms of their $A$-Davis-Wielandt radii. In addition, a connection between $A$-seminorm-parallelism to the identity operator and an equality condition for the $A$-Davis-Wielandt radius is proved. This generalizes the well-known results in \cite{zamanilma2018,chanchan}. Some other related results are also discussed.
\end{abstract}

\section{Introduction and Preliminaries}\label{s1}
Let $\big(\mathcal{H}, \langle\cdot\mid\cdot\rangle\big)$ be a non trivial complex equipped with the norm $\|\cdot\|$. Let $\mathcal{B}(\mathcal{H})$ denote the $C^{\ast}$-algebra of all bounded linear operators on $\mathcal{H}$ with identity $I_{\mathcal{H}}$ (or $I$ if no confusion arises). For $T\in\mathcal{B}(\mathcal{H})$, the classical numerical range of $T$ was introduced by Toeplitz in \cite{t1} as
$$W(T):=\{\langle T x\mid x\rangle;\;x \in \mathcal{H}\;\;\text{with}\;\|x\|=1\}.$$
For more information about this concept, the reader is invited to consult \cite{band,li,gus} and the references therein. Also, the numerical radius of $T\in \mathcal{B}(\mathcal{H})$ is given by
\begin{equation*}
\omega(T) = \sup \big\{|\mu|\,;\; \mu\in W(T)\big\}.
\end{equation*}
The concepts of numerical range and numerical radius play important roles in many different areas, especially mathematics and physics (see \cite{G.R, H.J,bakfeki02}). There is a rich variety of generalizations of the notion of the numerical range and the numerical radius. The reader may consult for example \cite{givens}. One of these generalizations is the Davis-Wielandt shell and radius of an operator $T\in\mathcal{B}(\mathcal{H})$ which are given by
 $$DW(T):=\left\{(\langle T x\mid x\rangle, \langle T x\mid Tx\rangle)\,;\;x \in \mathcal{H}\;\;\text{with}\;\|x\|=1\right\},$$
 and
 $$d\omega(T)=\displaystyle\sup\left\{\sqrt{|\langle Tx\mid x\rangle|^2+\|Tx\|^4}\,;\;x\in \mathcal{H},\;\|x\|=1\right\},$$
 respectively. The concept of the Davis-Wielandt shell is useful in studying operators. For more details see \cite{lipoonsze} and the references therein.

Let $\mathbb{T}$ denote the unit cycle of the complex plane, i.e. $\mathbb{T}=\{\lambda\in \mathbb{C}\,;\;|\lambda|=1 \}$. The notion of the norm-parallelism for Hilbert space operators has been introduced by A. Zamani et al. in \cite{zamanilma2015} as follows.
\begin{definition}
Let $T,S\in \mathcal{B}(\mathcal{H})$. We say that $T$ is norm-parallel to $S$, in short $T\parallel S$, if there exists $\lambda\in \mathbb{T}$ such that
$$\|T+\lambda S\|=\|T\|+\|S\|.$$
\end{definition}
 Some characterizations of the norm-parallelism for Hilbert space operators were given in \cite{zamanilma2015,Z.M.2}. In particular, we have the following useful theorem.
 \begin{theorem}\label{paranorm}(\cite[Corollary 2.12.]{Z.M.2})
Let $T,S\in \mathcal{B}(\mathcal{H})$. Then, the following assertions are equivalent:
\begin{itemize}
  \item [(1)] $T\parallel S$.
  \item [(2)] There exists a sequence $(x_n)_n\subset\mathcal{H}$ such that $\|x_n\|=1$ and
    \begin{equation*}
    \lim_{n\to \infty}|\langle T x_n\mid Sx_n\rangle|=\|T\|\|S\|.
    \end{equation*}
\end{itemize}
\end{theorem}
Recently, the relation between the norm-parallelism of operators and their Davis-Wielandt radii is discussed. In particular, we recall the following theorem.
\begin{theorem}\label{thm001}(\cite[Theorem 3.1]{zamanilma2018})
Let $T\in \mathcal{B}(\mathcal{H})$. Then, $T\parallel I$ if and only if $d\omega(T)=\sqrt{\omega(T)^2+\|T\|^4}$.
\end{theorem}

Recall from \cite{chanchan} that an operator $T\in \mathcal{B}(\mathcal{H})$ is said to be normaloid if $\omega(T)=\|T\|$. A characterization of normaloid operators is given in the following theorem.
\begin{theorem}\label{thm002}(\cite[Proposition 2.]{chanchan})
Let $T\in \mathcal{B}(\mathcal{H})$. Then, $T$ is normaloid if and only if $d\omega(T)=\sqrt{\omega(T)^2+\|T\|^4}$.
\end{theorem}
One main target of this paper is to extend Theorems \ref{paranorm} and \ref{thm002} to the context of semi-Hilbertian space operators.

 Recently, a new type of parallelism for Hilbert space operators based on numerical radius has been introduced by M. Mehrazin et al. in \cite{mehamzamani} as follows.
\begin{definition}\label{de.01}
An element $T\in\mathcal{B}(\mathcal{H})$ is called the numerical radius parallel
to another element $S \in\mathcal{B}(\mathcal{H})$, denoted by $T \parallel_{\omega} S$, if
\begin{align*}
\omega(T + \lambda S) = \omega(T)+\omega(S) \;\text{ for some }\;\lambda\in\mathbb{T}.
\end{align*}
\end{definition}

The following result gives a characterization of the numerical radius parallelism for Hilbert space operators. The proof can be found in \cite[Theorem 2.2]{mehamzamani}

\begin{theorem}\label{th.1}(\cite[Theorem 2.2]{mehamzamani})
Let $T,S\in \mathcal{B}(\mathcal{H})$. Then the following conditions are equivalent:
\begin{itemize}
\item[(i)] $T \parallel_{\omega} S$.
\item[(ii)] There exists a sequence of unit vectors $\{x_n\}$ in $\mathcal{H}$ such that
\begin{equation}\label{zamthm}
\lim_{n\rightarrow\infty} \big|\langle Tx_n\mid x_n\rangle\langle Sx_n\mid x_n\rangle\big| = \omega(T)\omega(S).
\end{equation}
\end{itemize}
In addition, if $\{x_n\}$ is a sequence of unit vectors in $\mathcal{H}$ satisfying \eqref{zamthm}, then it also satisfies $\displaystyle{\lim_{n\rightarrow\infty}}|\langle Tx_n\mid x_n\rangle| = \omega(T)$ and $\displaystyle{\lim_{n\rightarrow\infty}}|\langle Sx_n\mid x_n\rangle| = \omega(S)$.
\end{theorem}

 For the sequel, it is useful to point out the following facts. In all that follows, by an operator we mean a bounded linear operator. The range of every operator $T$ is denoted by $\mathcal{R}(T)$, its null space by $\mathcal{N}(T)$ and its adjoint by $T^*$. Let  $\mathcal{B}(\mathcal{H})^+$ be the cone of positive (semi-definite) operators, i.e.,
$\mathcal{B}(\mathcal{H})^+=\left\{A\in \mathcal{B}(\mathcal{H})\,;\,\langle Ax\mid x\rangle\geq 0,\;\forall\;x\in \mathcal{H}\;\right\}$. Any $A\in \mathcal{B}(\mathcal{H})^+$ defines a positive semi-definite sesquilinear form as follows:
$$\langle\cdot\mid\cdot\rangle_{A}:\mathcal{H}\times \mathcal{H}\longrightarrow\mathbb{C},\;(x,y)\longmapsto\langle x\mid y\rangle_{A} :=\langle Ax\mid y\rangle.$$
Notice that the induced semi-norm is given by $\|x\|_A=\langle x\mid x\rangle_A^{1/2},\;\forall\,x\in \mathcal{H}$. This makes $\mathcal{H}$ into a semi-Hilbertian space. One can check that $\|\cdot\|_A$ is a norm on $\mathcal{H}$ if and only if $A$ is injective, and that $(\mathcal{H},\|\cdot\|_A)$ is complete if and only if $\mathcal{R}(A)$ is closed. Further, $\langle\cdot\mid\cdot\rangle_{A}$ induces a semi-norm on a certain subspace of $\mathcal{B}(\mathcal{H})$ as follows. Given $T\in\mathcal{B}(\mathcal{H})$, if there exists $c>0$ satisfying $\|Tx \|_{A} \leq c \|x \|_{A}$, for all $x\in
\overline{\mathcal{R}(A)}$ it holds:
\begin{equation*}\label{semii}
\|T\|_A:=\sup_{\substack{x\in \overline{\mathcal{R}(A)},\\ x\not=0}}\frac{\|Tx\|_A}{\|x\|_A}=\displaystyle\sup_{\substack{x\in \overline{\mathcal{R}(A)},\\ \|x\|_A= 1}}\|Tx\|_{A}<\infty.
\end{equation*}
We denote $\mathcal{B}^{A}(\mathcal{H}):=\left\{T\in \mathcal{B}(\mathcal{H})\,;\,\|T\|_{A}< \infty\right\}$. It can be seen that $\mathcal{B}^{A}(\mathcal{H})$ is not a subalgebra of $\mathcal{B}(\mathcal{H})$.

From now on, we suppose that $A\neq0$ and we denote $P_A$ the orthogonal projection onto $\overline {\mathcal{R}(A)}$. Henceforth, $A$ is implicitly understood as a positive operator.

\begin{definition}(\cite{acg1})
For  $T \in \mathcal{B}(\mathcal{H})$, an operator $S\in\mathcal{B}(\mathcal{H})$ is called an $A$-adjoint of $T$ if for every $x,y\in \mathcal{H}$, we have
$\langle Tx\mid y\rangle_A=\langle x\mid Sy\rangle_A,$ i.e., $AS=T^*A.$
\end{definition}
The existence of an $A$-adjoint operator is not guaranteed. The set of all operators which admit $A$-adjoints is denoted by $\mathcal{B}_{A}(\mathcal{H})$. By Douglas' theorem \cite{doug}, one can verify that
$$\mathcal{B}_{A}(\mathcal{H})=\left\{T\in \mathcal{B}(\mathcal{H})\,;\;\mathcal{R}(T^{*}A)\subset \mathcal{R}(A)\right\},$$
and
$$\mathcal{B}_{A^{1/2}}(\mathcal{H})=\left\{T \in \mathcal{B}(\mathcal{H})\,;\;\exists \,c > 0;\;\|Tx\|_{A} \leq c \|x\|_{A},\;\forall\,x\in \mathcal{H}  \right\}.$$
If $T\in \mathcal{B}_{A^{1/2}}(\mathcal{H})$, we will say that $T$ is $A$-bounded. An important observation is that if $T$ is $A$-bounded, then $T(\mathcal{N}(A))\subset \mathcal{N}(A)$ and
\begin{equation}\label{ineqwich02}
\|Tx\|_{A} \leq\|T\|_{A} \|x\|_{A},\;\forall x\in \mathcal{H}.
\end{equation}
 Note that $\mathcal{B}_{A}(\mathcal{H})$ and $\mathcal{B}_{A^{1/2}}(\mathcal{H})$ are two subalgebras of
$\mathcal{B}(\mathcal{H})$ which are neither closed nor dense in $\mathcal{B}(\mathcal{H})$. Moreover, the following inclusions $\mathcal{B}_{A}(\mathcal{H})\subset\mathcal{B}_{A^{1/2}}(\mathcal{H})\subset\mathcal{B}^{A}(\mathcal{H})\subset \mathcal{B}(\mathcal{H})$ hold
with equality if $A$ is injective and has a closed range. In addition, for $T\in\mathcal{B}_{A^{1/2}}(\mathcal{H})$ we have:
\begin{equation}\label{seminormold}
\|T\|_A=\sup\left\{\|Tx\|_{A};\;x\in \mathcal{H},\,\|x\|_{A}= 1\right\}.
\end{equation}
For an account of results, we refer to \cite{bakfeki01,bakfeki04,bar,acg2,hsds} and the references therein.

Recently, the $A$-numerical radius of an operator $T\in\mathcal{B}(\mathcal{H})$ is defined by
$$\omega_A(T) = \sup\left\{|\langle Tx\mid x\rangle_A|\,;\;x\in \mathcal{H},\;\|x\|_A= 1\right\}.$$
This new concept is intensively studied (see \cite{bakfeki01,zamani2019}). Note that if $T \in \mathcal{B}(\mathcal{H})$ and satisfies $T(\mathcal{N}(A))\nsubseteq\mathcal{N}(A)$, then $\omega_A(T)=+\infty$ (see \cite[Theorem 2.2.]{kais01}). Moreover, it is easy to see that $\omega_A$ defines a seminorm on $\mathcal{B}_{A^{1/2}}(\mathcal{H})$. For more information about this concept the reader is invited to consult \cite{kais01,bakfeki01,zamani2019}  and the references therein. One main objective of this paper is to introduce a new type of parallelism for $A$-bounded operators based on the $A$-numerical radius and to extend Theorem \ref{th.1}.

If $T\in \mathcal{B}_A(\mathcal{H})$, the reduced solution of the equation $AX=T^*A$ is a distinguished $A$-adjoint operator of $T$, which is denoted by $T^\sharp$. Note that, $T^\sharp=A^\dag T^*A$ in which $A^\dag$ is the Moore-Penrose inverse of $A$. For more results concerning $T^\sharp$ see \cite{acg1,acg2}. From now on, to simplify notation, we will write $X^\sharp$ instead of $X^{\sharp_A}$ for every $X\in \mathcal{B}_A(\mathcal{H})$.

An operator $U\in \mathcal{B}_A(\mathcal{H})$ is called $A$-unitary if $\|Ux\|_A=\|U^\sharp x\|_A=\|x\|_A$ for all $x\in \mathcal{H}$. Further, we have $U^\sharp U=(U^\sharp)^\sharp U^\sharp=P_A$. Notice also that if $U$ is $A$-unitary, then $U^\sharp$ is $A$-unitary and $\|U\|_A=\|U^\sharp \|_A=1$. For more details about this class of operators we refer to \cite{acg1}. In recent years, several results covering some classes of operators on a complex Hilbert space $(\mathcal{H},\langle\cdot\mid\cdot\rangle)$ are extended to $(\mathcal{H},\langle\cdot\mid\cdot\rangle_A)$. One may see \cite{bakfeki01,bakfeki04,zamani2019} and their references.

The remainder of the paper is organized as follows. Section \ref{s2} is meant to introduce the concept of $A$-Davis-Wielandt shell of an operator $T$ acting on a complex Hilbert space $\mathcal{H}$ and to present some of its basic properties. In particular, unlike the classical Davis-Wielandt shell of operators, we will show that this new concept is in general unbounded. Mainly, by using some useful results related to the Hilbert space $\mathbf{R}(A^{1/2})$ we will prove that the $A$-Davis-Wielandt shell of $A$-bounded operators is bounded. Moreover, the convexity and the compactness of this new concept is studied. In addition, we introduce and investigate the notion of $A$-Davis-Wielandt radius of an operator $T$ denoted $d\omega_A(T)$. Mainly, we will show that an operator $T$ is $A$-normaloid if and only if $d\omega_A(T)=\|T\|_A\sqrt{1+\|T\|_A^2}$, where where $\omega_A(T)$ and $\|T\|_A$ denote respectively the $A$-numerical radius and the $A$-operator seminorm of $T$. In section \ref{s3}, we will introduce new concepts of parallelism in the framework of semi-Hilbertian spaces. More precisely, the parallelism of $A$-bounded operators with respect to $\|T\|_A$ and $\omega_A(T)$ are investigated. In particular, we will study the connection between the parallelism of $A$-bounded operators with respect to $\|T\|_A$ and the following equality $d\omega_A(T)=\|T\|_A\sqrt{1+\|T\|_A^2}$. Mainly, some characterizations of $A$-seminorm-parallelism to the identity operator are given.

 \section{The $A$-Davis-Wielandt shell of operators}\label{s2}
Motivated by theoretical study and applications of different generalizations of the numerical range. We introduce the concept of the $A$-Davis-Wielandt shell of an operator $T \in {\mathcal B}({\mathcal H})$ as follows.
\begin{definition}
Let $T\in\mathcal{B}(\mathcal{H})$. The $A$-Davis-Wielandt shell is defined as
$$DW_A(T):=\{(\langle T x\mid x\rangle_A, \langle T x\mid Tx\rangle_A)\,;\;x \in \mathcal{H}\;\;\text{with}\;\|x\|_A=1\}.$$
\end{definition}

\begin{remark}
If $T\in\mathcal{B}_A(\mathcal{H})$, then the $A$-Davis-Wielandt shell is given by
$$DW_A(T):=\{(\langle T x\mid x\rangle_A, \langle T^\sharp T x\mid x\rangle_A)\,;\;x \in \mathcal{H}\;\;\text{with}\;\|x\|_A=1\}.$$
Hence $DW_A(T)$ can be seen as the $A$-joint numerical range of $(T,T^\sharp T)$. For more details about the $A$-joint numerical range of a $d$-tuple of operators $\mathbf{T}=(T_1,\cdots, T_d)\in \mathcal{B}(\mathcal{H})^d$ the reader is invited to consult \cite{bakfeki01}.
\end{remark}

In the following proposition, we sum up some basic properties of the $A$-Davis-Wielandt shell of operators.
\begin{proposition}\label{prop2.1}
Let $T \in \mathcal{B}(\mathcal{H})$. Then, the following properties hold:
\begin{itemize}
\item [(1)] If $AT=TA$ and $\mathcal{R}(A)=\mathcal{H}$, than $DW_A(T)=DW(T)$.
\item [(2)] $DW_A(\alpha T+\beta I)=\left\{(\alpha \lambda+\beta,|\alpha|^2\mu+2Re(\alpha\overline{\beta}\lambda) +|\beta|^2)\,;\;(\lambda,\mu)\in DW_A(T)\right\}$, for every $\alpha,\beta \in \mathbb{C}$.
\item [(3)] If $T\in \mathcal{B}_A(\mathcal{H})$, then $ DW_A(T^\sharp)=\{(\overline{\lambda},\mu)\;;\;(\lambda,\mu)\in W_A(T)\times W_A(TT^\sharp)\,\}$.
\item [(4)] If $\mathcal{N}(A)$ is an invariant subspace for $T$, then $DW_A(UTU^\sharp)=DW_A(T)$ for any $A$-unitary operator $U$.
\end{itemize}
\end{proposition}
\begin{proof}
\noindent (1)\;Observe that if $AT=TA$ then $A^{1/2}T=TA^{1/2}$. So, we infer that
\begin{align*}
DW_A(T)
&=\left\{(\langle A^{1/2}T x\mid A^{1/2}x\rangle,\|A^{1/2}T x\|^2)\,;\;x \in \mathcal{H},\;\|A^{1/2}x\|=1\right\}\\
&=\left\{(\langle T A^{1/2}x\mid A^{1/2}x\rangle,\|TA^{1/2}x\|^2)\,;\;x \in \mathcal{H},\;\|A^{1/2}x\|=1\right\}\\
 &\subseteq DW(T).
\end{align*}
Further, the fact $\mathcal{R}(A)=\mathcal{H}$ implies that $A^{1/2}$ is a surjective operator. Hence, the reverse inclusion follows immediately. Thus, $DW_A(T)=DW(T)$ as required.
\par \vskip 0.1 cm \noindent (2)\;Can be verified readily.
\par \vskip 0.1 cm \noindent (3)\;If $(\lambda,\mu)\in DW_A(T^\sharp)$, then there exists $x\in {\mathcal H}$ with $\|x\|_A=1$ such that
$$\lambda=\langle T^\sharp x\mid x\rangle_A\;\;\text{and}\;\;\mu=\langle T^\sharp x\mid T^\sharp x\rangle_A.$$
This yields that $\lambda=\overline{\langle T x\mid x\rangle_A}$ or equivalently  $\overline{\lambda}=\langle T x\mid x\rangle_A \in W_A(T)$. On the other hand,
$\mu=\langle T T^\sharp x\mid  x\rangle_A \in  W_A(TT^\sharp)$. Consequently,
$$ DW_A(T^\sharp)\subset \{(\overline{\lambda},\;\mu)\,;\;(\lambda,\mu)\in W_A(T)\times W_A(TT^\sharp)\;\}.$$
By a similar way we prove the reverse inclusion.
 \par \vskip 0.1 cm \noindent (4)\;Let  $(\lambda,\mu)\in DW_A(T)$, then there exists $x\in {\mathcal H}$ with $\|x\|_A=1$ such that
$$\lambda = \langle T x\mid x\rangle_A\;\;\;\text{and}\;\;\mu=\langle T x\mid Tx\rangle_A=\langle UT x\mid UTx\rangle_A.$$
 By writing $x=P_Ax+y$ with $y\in \mathcal{N}(A)$ and by using the fact that $\mathcal{N}(A)$ is an invariant subspace for $T$, we see that
 \begin{align*}
\lambda
& = \langle AT(P_Ax+y)\mid P_Ax+y\rangle\\
 &=\langle AT(P_Ax+y),P_Ax\rangle\\
 &= \langle AT(P_Ax)\mid P_Ax\rangle \\
  &=\langle TP_Ax\mid P_Ax\rangle_A.
\end{align*}
Similarly, we verify that
$$\mu=\langle UTP_Ax\mid UTP_Ax\rangle_A.$$
Moreover, since $P_A=U^\sharp U$, it follows that
$$\lambda=\langle TU^\sharp Ux\mid U^\sharp Ux\rangle_A=\langle U TU^\sharp Ux\mid Ux\rangle_A,$$
and
$$\mu=\langle UTU^\sharp Ux\mid UTU^\sharp Ux\rangle_A.$$
Hence, since $\|U x\|_A=\| x\|_A=1$, then we immediately deduce that $(\lambda,\mu) \in DW_A(UTU^\sharp)$ and so $DW_A(T)\subseteq DW_A(UTU^\sharp)$. Finally, the reverse inclusion can be checked by using the fact that $\|U^\sharp y\|_A=\|y\|_A$ for all $y\in \mathcal{H}$.
\end{proof}

It should be mentioned that $DW(T)$ is bounded for every $T\in \mathcal{B}(\mathcal{H})$ (see \cite{lipoonsze}). However, for an arbitrary operator $T$, the $A$-Davis-Wielandt shell of $T$ is in general unbounded even if $\mathcal{H}$ is finite dimensional as it is shown in the following example.

\begin{example}\label{exampleunbounded}
Let $A = \begin{pmatrix}1&0&0\\0&0&0\\0&0&0\end{pmatrix}$ and $T= \begin{pmatrix}0&0&1\\0&1&0\\1&0&0\end{pmatrix}$ be operators on $\mathbb{C}^3$. By a direct computation we get
$$
DW_A(T)=\left\{\left(x_2\overline{x_1},1\right);\,(x_1,x_2) \in \mathbb{C}^2\right\}=\mathbb{C}\times \{1\},
$$
which is clearly unbounded.
\end{example}
Unlike the case $T\in \mathcal{B}(\mathcal{H})$, we have $DW_A(T)$ is bounded if $T\in \mathcal{B}_{A^{1/2}}(\mathcal{H})$ as it will be shown in the next result. Before that, we need to recall from \cite{acg3} the following facts. The semi-inner product $\langle\cdot\mid\cdot\rangle_A$ induces an inner product on the quotient space $\mathcal{H}/\mathcal{N}(A)$ defined as
$$[\overline{x},\overline{y}]=\langle Ax\mid y\rangle,$$
for all $\overline{x},\overline{y}\in \mathcal{H}/\mathcal{N}(A)$. Notice that $(\mathcal{H}/\mathcal{N}(A),[\cdot,\cdot])$ is not complete unless $\mathcal{R}(A)$ is not closed. However, a canonical construction due to L. de Branges and J. Rovnyak in \cite{branrov} shows that the completion of $\mathcal{H}/\mathcal{N}(A)$  under the inner product $[\cdot,\cdot]$ is isometrically isomorphic to the Hilbert space $\mathcal{R}(A^{1/2})$
with the inner product
$$(A^{1/2}x,A^{1/2}y):=\langle P_Ax\mid P_Ay\rangle,\;\forall\, x,y \in \mathcal{H}.$$

In the sequel, the Hilbert space $\left(\mathcal{R}(A^{1/2}), (\cdot,\cdot)\right)$ will be denoted by $\mathbf{R}(A^{1/2})$ and we use the symbol
$\|\cdot\|_{\mathbf{R}(A^{1/2})}$ to represent the norm induced by the inner product $(\cdot,\cdot)$. It is worth noting that $\mathcal{R}(A)$ is dense
in $\mathbf{R}(A^{1/2})$ (see \cite{kais01}). Moreover, since $\mathcal{R}(A)\subset \mathcal{R}(A^{1/2})$, one observes that
\begin{align}\label{usefuleq001}
(Ax,Ay)
&=(A^{1/2}A^{1/2}x,A^{1/2}A^{1/2}y) = \langle P_AA^{1/2}x\mid P_AA^{1/2}y\rangle\nonumber\\
 &=\langle A^{1/2}x\mid A^{1/2}y\rangle=\langle x \mid y\rangle_A.
\end{align}
This gives the following important relation:
\begin{equation}\label{usefuleq01}
\|Ax\|_{\mathbf{R}(A^{1/2})}=\|x\|_A,\;\forall\,x\in \mathcal{H}.
\end{equation}
For more information related to the Hilbert space $\mathbf{R}(A^{1/2})$, the interested reader is referred to \cite{acg3} and the references therein.

As in \cite{acg3}, we consider the operator $Z_A: \mathcal{H}\to \mathbf{R}(A^{1/2})$ defined by
\begin{equation}
Z_Ax=Ax,\;\forall\,x\in \mathcal{H}.
\end{equation}
Now, we adapt from \cite{acg3} the following proposition in our context.
\begin{proposition}\label{prop_arias}
Let $T\in \mathcal{B}(\mathcal{H})$. Then $T\in \mathcal{B}_{A^{1/2}}(\mathcal{H})$ if and only if there exists a unique $\widetilde{T}\in \mathcal{B}(\mathbf{R}(A^{1/2}))$ such that $Z_AT =\widetilde{T}Z_A$.
\end{proposition}
Before we move on, it is important state the following lemma. Its proof can be found in the proof \cite[Proposition 3.9.]{kais01}.
\begin{lemma}\label{imporeq2009}
If $T\in \mathcal{B}_{A^{1/2}}(\mathcal{H})$, then we have
\begin{itemize}
  \item [(1)] $\|T\|_A=\|\widetilde{T}\|_{\mathcal{B}(\mathbf{R}(A^{1/2}))}$.
  \item [(2)] $\omega_A(T)=\omega(\widetilde{T})$.
\end{itemize}
\end{lemma}

In order to prove our next result, we need the following lemma.
\begin{lemma}\label{norminraundemi01}
Let $T\in \mathcal{B}_{A^{1/2}}(\mathcal{H})$. Then,
\begin{equation}\label{eq90}
\overline{DW_A(T)}=\overline{DW(\widetilde{T})},
\end{equation}
where $DW(\widetilde{T})$ denotes the classical Davis-Wielandt shell of $\widetilde{T}$ on the Hilbert space $\mathbf{R}(A^{1/2})$.
\end{lemma}

\begin{proof}
Since $\mathcal{H}=\mathcal{N}(A^{1/2})\oplus \overline{\mathcal{R}(A^{1/2})}$, then we see that
\begin{align}\label{newidea}
DW_A(T)
& =\left\{\left(\langle Tx\mid x\rangle_A,\|Tx\|_A^2\right)\,;\;x\in \overline{\mathcal{R}(A^{1/2})},\;\|x\|_A= 1\right\} \nonumber\\
 &=\left\{ \left[(ATx, Ax),\|ATx\|_{\mathbf{R}(A^{1/2})}^2\right]\,;\;x\in \overline{\mathcal{R}(A^{1/2})},\;\|Ax\|_{\mathbf{R}(A^{1/2})}= 1\right\}\nonumber\\
  &=\left\{\left[( \widetilde{T}Ax, Ax),\|\widetilde{T}Ax\|_{\mathbf{R}(A^{1/2})}^2\right]\,;\;x\in \overline{\mathcal{R}(A^{1/2})},\;\|Ax\|_{\mathbf{R}(A^{1/2})}= 1\right\}\\
  &\subseteq DW(\widetilde{T}).\nonumber
\end{align}
This implies that, $\overline{DW_A(T)}\subseteq\overline{DW(\widetilde{T})}$. On the other hand, one has
\begin{align*}
DW(\widetilde{T})
& =\left\{\left[(\widetilde{T}y, y),\|\widetilde{T}y\|_{\mathbf{R}(A^{1/2})}^2\right]\,;\;y\in \mathcal{R}(A^{1/2}),\;\|y\|_{\mathbf{R}(A^{1/2})}= 1\right\}\\
& =\left\{\left[( \widetilde{T}A^{1/2}x, A^{1/2}x),\|\widetilde{T}A^{1/2}x\|_{\mathbf{R}(A^{1/2})}^2\right]\,;\;x\in \mathcal{H},\;\|A^{1/2}x\|_{\mathbf{R}(A^{1/2})}= 1\right\}\\
& =\left\{\left[( \widetilde{T}A^{1/2}x, A^{1/2}x),\|\widetilde{T}A^{1/2}x\|_{\mathbf{R}(A^{1/2})}^2\right]\,;\;x\in \overline{\mathcal{R}(A^{1/2})},\;\|A^{1/2}x\|_{\mathbf{R}(A^{1/2})}= 1\right\}.
\end{align*}
Since $\mathcal{R}(A)$ is dense in $\mathbf{R}(A^{1/2})$, then for all $x\in \mathcal{H}$, there exists a sequence $(x_n)_n\subset \mathcal{H}$ (which can chosen to be in $\overline{\mathcal{R}(A^{1/2})}$ because of the decomposition $\mathcal{H}=\mathcal{N}(A^{1/2})\oplus \overline{\mathcal{R}(A^{1/2})}$) such that $$\displaystyle\lim_{n\to +\infty}\|Ax_n-A^{1/2}x\|_{\mathbf{R}(A^{1/2})}=0.$$
So, if $(\lambda,\mu)\in DW(\widetilde{T})$ then there exists $(x_n)_n\subseteq\overline{\mathcal{R}(A^{1/2})}$ such that
$$\displaystyle\lim_{n\to +\infty}\|Ax_n\|_{\mathbf{R}(A^{1/2})}= 1,\;\displaystyle\lim_{n\to +\infty}|(\widetilde{T}Ax_n, Ax_n)|=\lambda\text{ and } \displaystyle\lim_{n\to +\infty}\|\widetilde{T}Ax_n\|_{\mathbf{R}(A^{1/2})}^2=\mu.$$
Let $y_n:=\frac{x_n}{\|Ax_n\|_{\mathbf{R}(A^{1/2})}}$, then we see that $\|Ay_n\|_{\mathbf{R}(A^{1/2})}= 1$,
$$\lim_{n\to +\infty}|( \widetilde{T}Ay_n, Ay_n)|=\lambda\;\text{ and }\;\displaystyle\lim_{n\to +\infty}\|\widetilde{T}Ay_n\|_{\mathbf{R}(A^{1/2})}^2=\mu.$$
This implies, by using \eqref{newidea}, that $(\lambda,\mu)\in \overline{DW_A(T)}$ and so $DW(\widetilde{T})\subseteq \overline{DW_A(T)}$. Thus, we infer that $\overline{DW(\widetilde{T})}\subseteq \overline{DW_A(T)}$. Hence, \eqref{eq90} is proved and therefore the proof is complete.
\end{proof}

Now, we are in a position to prove the following theorem.
\begin{theorem}
Let $T\in \mathcal{B}_{A^{1/2}}(\mathcal{H})$. Then, $DW_A(T)$ is bounded. More precisely, we have $DW_A(T)\subseteq \Lambda$, where
$$\Lambda=\{(\mu,r)\in \mathbb{C}\times [0,+\infty[\,;\; |\mu|^2\leq r\leq \|T\|_A^2\,\}.$$
\end{theorem}
\begin{proof}
Since $T\in \mathcal{B}_{A^{1/2}}(\mathcal{H})$, then by Proposition \ref{prop_arias}, there exists a unique $\widetilde{T}\in \mathcal{B}(\mathbf{R}(A^{1/2}))$ such that $Z_AT =\widetilde{T}Z_A$. Moreover, by \cite[Theorem 2.1.]{lipoonsze}, $DW(\widetilde{T})$ is bounded and we have
$$DW(\widetilde{T})\subseteq \left\{(\mu,r)\in \mathbb{C}\times [0,+\infty[\,;\; |\mu|^2\leq r\leq \|\widetilde{T}\|_{\mathcal{B}(\mathbf{R}(A^{1/2}))}^2\,\right\}.$$
This implies that $DW(\widetilde{T})\subseteq \Lambda$ since in view of Lemma \ref{imporeq2009} we have $\|T\|_A=\|\widetilde{T}\|_{\mathcal{B}(\mathbf{R}(A^{1/2}))}$. On the other hand, by using Lemma \ref{norminraundemi01}, we infer that $\overline{DW_A(T)}\subseteq \Lambda$ since $\Lambda$ is closed. Hence, $DW_A(T)\subseteq \Lambda$ as required.
\end{proof}

Now, we turn to investigate the convexity and the compactness of $DW_A(T)$. Firstly, we recall from \cite{lipoonsze} the following theorem which remains true for operators acting on an inner-product space which is not necessarily complete.
\begin{theorem}\label{convexdavis}
Let $T\in \mathcal{B}(\mathcal{H})$ with $\text{dim}(\mathcal{H})\geq 3$. Then, $DW(T)$ is convex subset of $\mathbb{C}^2$.
\end{theorem}
The following result extends the above theorem to the case of $A$-bounded operators.

\begin{theorem}\label{mainold}
Let $T\in \mathcal{B}_{A^{1/2}}(\mathcal{H})$ with $\text{dim}(\mathcal{H})\geq 3$. Then, $DW_A(T)$ is a convex subset of $\mathbb{C}^2$.
\end{theorem}
\begin{proof}
Without loss of generality, we assume that $A$ is not injective (if $A$ is an injective operator, then the result follows trivially by applying Theorem \ref{convexdavis}). Since $\mathcal{H}=\mathcal{N}(A)\oplus \overline{\mathcal{R}(A)}$, then every $x\in\mathcal{H}$ can be written into $x = x_1 + x_2$ with $x_1\in\mathcal{N}(A)$ and $x_2\in \overline{\mathcal{R}(A)}$. Moreover, the fact that $\mathcal{N}(A)=\mathcal{N}(A^{1/2})$ implies that $\|x\|_A=\|x_2\|_A$. So, one has:
\begin{multline*}
DW_A(T)
 = \Bigl\{\left(\langle Tx_1\mid x_2\rangle_A+\langle Tx_2\mid x_2\rangle_A,\|Tx_1+Tx_2\|_A^2\right)\,;\;x_1 \in \mathcal{N}(A),\;x_2 \in \overline{\mathcal{R}(A)},\\
\;\;\text{with}\;\|x_2\|_A=1\;\Bigr\}.
\end{multline*}
On the other hand, since $T\in \mathcal{B}_{A^{1/2}}(\mathcal{H})$, then $T(\mathcal{N}(A))\subset\mathcal{N}(A)$. So, by using the fact that $\langle T x_2\mid x_2\rangle_A = \langle A P_AT x_2\mid x_2\rangle$, we obtain:
$$
DW_A(T)
 = \Bigl\{\left(\langle P_ATx_2\mid x_2\rangle_A,\|P_ATx_2\|_A^2\right)\,;\;x_2 \in \overline{\mathcal{R}(A)},\;\|x_2\|_A=1\;\Bigr\}=DW_{A_0}(T_0),
$$
 where $A_0=A|_{\overline{\mathcal{R}(A)}}$ and $T_0=P_AT|_{\overline{\mathcal{R}(A)}}$. Since $A_0$ is injective on $\overline{\mathcal{R}(A)}$ and $T_0\in \mathcal{B}(\overline{\mathcal{R}(A)})$, then by Theorem \ref{convexdavis} we infer that $DW_{A_0}(T_0)$. Therefore, $DW_A(T)$ is convex as desired.
\end{proof}
\begin{remark}
If $T(\mathcal{N}(A))\nsubseteq\mathcal{N}(A)$ and $\text{dim}(\mathcal{H})\geq 3$, then $DW_A(T)$ is not in general a convex subset of $\mathbb{C}^2$ as it is shown in the following example.
\end{remark}
\begin{example}\label{exampleunbounded}
Let $A = \begin{pmatrix}0&0&0\\0&0&0\\0&0&1\end{pmatrix}$ and $T= \begin{pmatrix}0&0&1\\0&1&0\\1&0&0\end{pmatrix}$ be operators on $\mathbb{C}^3$. A simple calculation shows that $T(\mathcal{N}(A))\nsubseteq\mathcal{N}(A)$ and
\begin{align*}
DW_A(T)
& =\left\{\left(a\overline{b},|b|^2\right);\,(a,b) \in \mathbb{C}^2,\;|b|=1\right\}.
\end{align*}
Moreover, it is not difficult to observe that
\begin{align*}
DW_A(T)
 &=\left\{\left(z,|z|^2\right);\,z\in \mathbb{C}\right\},
\end{align*}
which is not convex. Indeed, clearly we have $(1,1)\in DW_A(T)$ and $(i,1)\in DW_A(T)$ but obviously $(\tfrac{1+i}{2},1)\notin DW_A(T)$.
\end{example}

The following lemma is taken from \cite{bakfeki01}. For the reader's convenience, here we give a proof.
\begin{lemma}\label{lemmhomo}
Let $(E,\langle \cdot\mid \cdot\rangle)$ be a complex inner product space (not necessarily complete) with the associated norm $\|\cdot\|$
and let $A$ be a positive injective operator on $E$. Then, $S(0,1)$ and $S^A(0,1)$ are homeomorphic\index{Homeomorphic}, where $S(0,1)$ and $S^A(0,1)$ are defined as
$$S(0,1):=\{x\in E\,;\;\|x\|=1\}\quad\text{and}\quad S^A(0,1):=\{x\in E\,;\;\|x\|_A=1\}.$$
\end{lemma}
\begin{proof}
Let us consider the following maps
\begin{equation*}
\begin{aligned}[t]
  h\colon S(0,1) &\rightarrow S^A(0,1)\\
  x              &\mapsto \frac{x}{\|x\|_A},
\end{aligned}
\quad\text{and}\quad
\begin{aligned}[t]
  g\colon S^A(0,1) &\rightarrow S(0,1)\\
  y                &\mapsto \frac{y}{\|y\|},
\end{aligned}
\end{equation*}
Since $A$ is injective, then $\|x\|_A=0$ if and only if $x=0$. This ensure that $h$ is well-defined. Furthermore, it can be seen that $h$ and $g$
are continuous and inverse of each other. Thus, we deduce that $S^A(0,1)$ and $S(0,1)$ are homeomorphic.
\end{proof}
Now, we are able to prove the following theorem.
\begin{theorem}\label{compactnew}
Let $\mathcal{H}$ be a finite dimensional complex Hilbert space and $A\in \mathcal{B}(\mathcal{H})^+$ be such that $T(\mathcal{N}(A))\subseteq\mathcal{N}(A)$. Then $DW_A(T)$ is a compact subset of $\mathbb{C}^2$.
\end{theorem}
\begin{proof}
By proceeding as in the proof of Theorem \ref{mainold}, we see that $DW_A(T)=DW_{A_0}(T_0)$, with $A_0=A|_{\overline{\mathcal{R}(A)}}$ and $T_0=P_AT|_{\overline{\mathcal{R}(A)}}\in \mathcal{B}(\overline{\mathcal{R}(A)})$. Moreover, we set
$$\widetilde{S}(0,1):=\left\{x\in \overline{\mathcal{R}(A)}\,;\,\|x\|=1\right\}\;\;\text{and}\;\;S^{A_0}(0,1):=\left\{x\in \overline{\mathcal{R}(A)}\,;\,\|x\|_{A_0}=1\right\}.$$
By Lemma \ref{lemmhomo}, $S^{A_0}(0,1)$ is homeomorphic to $\widetilde{S}(0,1)$ which is compact since $\overline{\mathcal{R}(A)}$ is finite dimensional. Now, let us consider the following map
$$\psi: S^{A_0}(0,1)\to \mathbb{C}^2: x \mapsto (\langle T_0x\mid x\rangle_{A_0},\|T_0x\|_{A_0}^2).$$
It can observed that $\psi$ is continuous. This implies that $DW_{A_0}(T_0)$ is compact. Therefore, $DW_A(T)$ is compact as required.
\end{proof}
The following corollary is a direct consequence of Theorem \ref{compactnew}.
\begin{corollary}
Let $\mathcal{H}$ be a finite dimensional complex Hilbert space. If $T\in\mathcal{B}_{A^{1/2}}(\mathcal{\mathcal{H}})$, then $DW_A(T)$ is a compact subset of $\mathbb{C}^2$.
\end{corollary}

Now, we deal with the $A$-Davis-Wielandt radius of operators in semi-Hilbert spaces.
\begin{definition}
Let $T\in\mathcal{B}(\mathcal{H})$. The $A$-Davis-Wielandt radius of $T$ is given by
\begin{align*}
d\omega_A(T)
&=\displaystyle\sup\left\{\|\lambda\|_2,\;\lambda=(\lambda_1,\lambda_2)\in DW_A(T)\right\}\\
&=\displaystyle\sup\left\{\sqrt{|\langle Tx\mid x\rangle_A|^2+\|Tx\|_A^4}\,;\;x\in \mathcal{H},\;\|x\|_A=1\right\}.
\end{align*}
\end{definition}

\begin{remark}
\noindent (1)\;Clearly, $d\omega_A(T)\geq0$ for every $T\in\mathcal{B}(\mathcal{H})$.
\par \vskip 0.1 cm \noindent (2)\;If $T\in\mathcal{B}(\mathcal{H})$, then $d\omega_A(T)$ can be equal to $+\infty$. Indeed, by using the same operators as in Example \ref{exampleunbounded}, it can be seen that $d\omega_A(T)=+\infty$.
\par \vskip 0.1 cm \noindent (3)\;If $T\in\mathcal{B}_A(\mathcal{H})$, then the $A$-Davis-Wielandt radius of $T$ is given by
$$d\omega_A(T)=\displaystyle\sup\left\{\sqrt{|\langle Tx\mid x\rangle_A|^2+|\langle T^\sharp T x\mid x\rangle_A|^2}\,;\;x\in \mathcal{H},\;\|x\|_A=1\right\}.$$
\end{remark}

In the following proposition, we summarize some basic properties of the $A$-Davis-Wielandt radius of operators.
\begin{proposition}\label{propositionsum}
Let $T,S\in\mathcal{B}_{A^{1/2}}(\mathcal{\mathcal{H}})$. Then, the following properties hold:
\begin{itemize}
\item [(1)] $d\omega_A(T)=0$ if and only if $AT=0$.
\item [(2)] $d\omega_A(UTU^\sharp)=d\omega_A(T)$ for any $A$-unitary operator $U\in \mathcal{B}_A(\mathcal{H})$.
\item [(3)] For all $\mu\in \mathbb{C}$ we have
$$
d\omega_A(\mu T)
\begin{cases}
\geq |\mu|d\omega_A(T)&,\text{ if }\;|\mu|> 1,\\
=d\omega_A(T)&,\text{ if }\;|\mu|=1,\\
\leq|\mu|d\omega_A(T)&,\text{ if }\;|\mu|< 1.
\end{cases}.
$$
\item [(4)] Suppose that $|\mu|\neq1$, then $d\omega_A(\mu T)=|\mu|d\omega_A(T)$ if and only if $AT=0$ or $\mu=0$.
\item [(5)] $d\omega_A(\cdot)$ satisfies
$$\max\{\omega_A(T),\|T\|_A^2\}\leq d\omega_A(T) \leq \sqrt{\omega_A(T)^2+\|T\|_A^4}.$$
\item [(6)]  $d\omega_A(\cdot)$ satisfies
$$d\omega_A(T+S)\leq \sqrt{2(d\omega_A(T)+d\omega_A(S))+4(d\omega_A(T)+d\omega_A(S))^2}.$$
\end{itemize}
\end{proposition}
\begin{proof}
\noindent (1)\;If $AT=0$, then clearly $d\omega_A(T)=0$. Conversely, assume that $d\omega_A(T)=0$. This implies that $\langle Tx\mid x\rangle_A=0$ for all $x\in S^A(0,1)$, where $S^A(0,1)$ is denoted to be the $A$-unit sphere of $\mathcal{H}$. Let $y\in \mathcal{H}\setminus  \mathcal{N}(A)$, then $\frac{y}{\|y\|_A}\in S^A(0,1)$ and we have $\langle Ty\mid y\rangle_A=0$. If $y\in \mathcal{N}(A)$, then clearly $\langle Ty\mid y\rangle_A=0$. Hence, we deduce that $\langle Ty\mid y\rangle_A=0$ for all $y\in \mathcal{H}$. So, $AT=0$.
\par \vskip 0.1 cm \noindent (2)\;Follows immediately from Proposition \ref{prop2.1}.
\par \vskip 0.1 cm \noindent (3)\;The case $|\mu|=1$ is obvious. Let $T\in\mathcal{B}_{A^{1/2}}(\mathcal{\mathcal{H}})$ be such that $AT\neq 0$. Let also $\mu\in \mathbb{C}$ be such that $|\mu|> 1$. By definition of $d\omega_A(\mu T)$ there exists a sequence $(x_n)_n\subset \mathcal{H}$ such that $\|x_n\|_A=1$ and
\begin{equation}\label{contradiction}
\lim_{n\to +\infty}\sqrt{|\langle \mu Tx_n\mid x_n\rangle_A|^2+\|\mu Tx_n\|_A^4}=d\omega_A(\mu T)>0.
\end{equation}
We claim that there exists $i\in \mathbb{N}^*$ such that $\|Tx_n\|_A>0$ for all $n>i$, where $\mathbb{N}^*$ is the set of positive natural numbers. Otherwise, assume that for every $i\in \mathbb{N}^*$, there exists a positive integer $n(i)$ such that $n(i) > i$ such that $\|Tx_{n(i)}\|_A = 0$. Therefore there is a subsequence $(x_{n_k})_k$ of $(x_n)_n$ such that
\begin{equation*}
\lim_{k\to +\infty}\sqrt{|\langle \mu Tx_{n_k}\mid x_{n_k}\rangle_A|^2+\|\mu Tx_{n_k}\|_A^4}=0.
\end{equation*}
This leads to a contradiction with \eqref{contradiction}. So, our claim is true. Hence, for all $n>i$ we see that
\begin{align*}
|\mu |^2\left(|\langle Tx_n\mid x_{n}\rangle_A|^2+\|Tx_n\|_A^4\right)
&=|\mu|^2|\langle Tx_n\mid x_{n}\rangle_A|^2+|\mu |^4\|Tx_n\|_A^4-\left(|\mu|^4-|\mu|^2\right)\|Tx_n\|_A^4\\
 &\leq |\mu|^2|\langle Tx_n\mid x_{n}\rangle_A|^2+|\mu |^4\|Tx_n\|_A^4\\
 &\leq  d\omega_A(\mu T)^2.
\end{align*}
Thus, by taking limit and square root, we see that $|\mu|d\omega_A(T)\leq  d\omega_A(\mu T)$. If $AT=0$, then the above result follows trivially.

Now, let $\nu\in \mathbb{C}$ be such that $0<|\nu|<1$. Let also $\mu=\frac{1}{\nu}$ and $S=\nu T$. By applying the previous result we infer that
$$|\mu|d\omega_A(S)\leq d\omega_A(\mu S).$$
This implies that $d\omega_A(\nu T)\leq |\nu|d\omega_A(T)$. If $\nu=0$, the the above result follows obviously.
\par \vskip 0.1 cm \noindent (4)\;Can be easily checked.
\par \vskip 0.1 cm \noindent (5)\;For any $x\in \mathcal{H}$ such that $\|x\|_A=1$ we have
\begin{align*}
\|Tx\|_A^2
&\leq \sqrt{|\langle Tx\mid x\rangle_A|^2+\|Tx\|_A^4}\\
&\leq \displaystyle\sup\left\{\sqrt{|\langle Ty\mid y\rangle_A|^2+\|Ty\|_A^4}\,;\;y\in \mathcal{H},\;\|y\|_A=1\right\}\\
 &=d\omega_A(T).
\end{align*}
 This shows by taking the supremum over all $x\in \mathcal{H}$ with $\|x\|_A=1$ and using \eqref{seminormold}, that
 \begin{equation} \label{dav1}
\|T\|_A^2\leq d\omega_A(T).
 \end{equation}
By proceeding as above, one shows that
 \begin{equation}\label{dav2}
\omega_A(T)\leq d\omega_A(T).
 \end{equation}
 So, combining \eqref{dav1} together with \eqref{dav2} yields
  \begin{equation}\label{dav02}
\max\{\omega_A(T),\|T\|_A^2\}\leq d\omega_A(T).
  \end{equation}
On the other hand, for any $x\in \mathcal{H}$ with $\|x\|_A=1$ we have
\begin{align*}
\sqrt{|\langle Tx\mid x\rangle_A|^2+\|Tx\|_A^4}
&\leq \displaystyle\sup\left\{\sqrt{|\langle Ty\mid y\rangle_A|^2+\|Ty\|_A^4}\,;\;y\in \mathcal{H},\;\|y\|_A=1\right\}\\
&\leq \sqrt{\omega_A(T)^2+\|T\|_A^4}.
\end{align*}
 This shows by taking the supremum over all $x\in \mathcal{H}$ with $\|x\|_A=1$, that
 \begin{equation} \label{dav3}
d\omega_A(T)\leq \sqrt{\omega_A(T)^2+\|T\|_A^4}.
 \end{equation}
 So, by combining \eqref{dav3} together with \eqref{dav02} we get the desired result.
\par \vskip 0.1 cm \noindent (6)\;It is well known that $\omega_A(T)\leq \|T\|_A$ for all $T\in\mathcal{B}_{A^{1/2}}(\mathcal{\mathcal{H}})$. This implies, by using \eqref{dav3}, that
\begin{align*}
d\omega_A(T+S)
&\leq \sqrt{\omega_A(T+S)^2+\|T+S\|_A^4}\\
&\leq \sqrt{\|T+S\|_A^2+\|T+S\|_A^4} \\
&\leq \sqrt{(\|T\|_A+\|S\|_A)^2+(\|T\|_A+\|S\|_A)^4}.
\end{align*}
On the other hand, by using the arithmetic-geometric mean inequality and \eqref{dav1}, we obtain
\begin{align*}
(\|T\|_A+\|S\|_A)^2
&=\|T\|_A^2+2\|T\|_A\|S\|_A+\|S\|_A^2\\
&\leq 2\|T\|_A^2+2\|S\|_A^2 \\
&\leq 2(d\omega_A(T)+d\omega_A(S)).
\end{align*}
So, the desired estimate is proved.
\end{proof}

\begin{remark}
\begin{itemize}
  \item [(1)] The inequalities in the assertion $(5)$ of the above proposition are sharp even if $A=I$ (see \cite{zamanilma2018}).
  \item [(2)] In view of the assertions $(1)$ and $(6)$ of the above proposition, one can observe that $d\omega_A(\cdot)$ is neither a norm nor a seminorm on $\mathcal{B}_{A^{1/2}}(\mathcal{H})$.
\end{itemize}
\end{remark}

Recently, K.Feki introduced in \cite{kais01} the concept of $A$-normaloid operators as follows.
\begin{definition}
An operator $T\in \mathcal{B}_{A^{1/2}}(\mathcal{H})$ is said to be $A$-normaloid if $r_A(T)=\|T\|_A$, where
$$r_A(T)=\displaystyle\lim_{n\to\infty}\|T^n\|_A^{\frac{1}{n}}.$$
\end{definition}

The following proposition gives some characterizations of $A$-normaloid operators.
\begin{proposition}(\cite{kais01})\label{charactenormailoid}
Let $T\in \mathcal{B}_{A^{1/2}}(\mathcal{H})$. Then, the following assertions are equivalent:
\begin{itemize}
  \item [(1)] $T$ is $A$-normaloid.
  \item [(2)] $\|T^n\|_A=\|T\|_A^n$ for all $n\in \mathbb{N}^*$.
  \item [(3)] $\omega_A(T)=\|T\|_A$.
    \item [(4)] There exists a sequence $(x_n)_n\subset\mathcal{H}$ such that $\|x_n\|_A=1$,
    \begin{equation*}
    \lim_{n\to \infty}\|T x_n\|_A=\|T\|_A \;\text{  and  }\; \lim_{n\to \infty}|\langle T x_n\mid x_n\rangle_A|=\omega_A(T).
    \end{equation*}
\end{itemize}
\end{proposition}

Now, we aim to characterize $A$-normaloid operators in terms of their Davis-Wielandt radius. We need the following lemma.
\begin{lemma}\label{norminraundemi}
Let $T\in \mathcal{B}_{A^{1/2}}(\mathcal{H})$. Then, $d\omega_A(T)=d\omega(\widetilde{T})$.
\end{lemma}
\begin{proof}
In view of Lemma \ref{norminraundemi01}, we have $\overline{DW_A(T)}=\overline{DW(\widetilde{T})}$. This yields that $d\omega_A(T)=d\omega(\widetilde{T})$ since $\sup \overline{\Omega}=\sup \Omega$ for every subset $\Omega$ of $\mathbb{C}^2$.
\end{proof}

We are now in a position to establish one of our main results. Our result generalizes Theorem \ref{thm002}.

\begin{theorem}\label{equinew}
Let $T\in \mathcal{B}_{A^{1/2}}(\mathcal{H})$. Then, the following assertions are equivalent:
\begin{itemize}
  \item [(1)] $T$ is $A$-normaloid.
  \item [(2)] $d\omega_A(T)=\sqrt{\omega_A(T)^2+\|T\|_A^4}$.
\end{itemize}
\end{theorem}

\begin{proof}
Notice first that, it was shown in \cite{kais01} that an operator $T\in \mathcal{B}_{A^{1/2}}(\mathcal{H})$ is $A$-normaloid if and only if $\widetilde{T}\in \mathcal{B}(\mathbf{R}(A^{1/2}))$ is a normaloid operator on the Hilbert space $\mathbf{R}(A^{1/2})$. So, by Theorem \ref{thm002}, $\widetilde{T}$ is normaloid if and only if $d\omega(\widetilde{T})=\sqrt{\omega(\widetilde{T})^2+\|T\|_{\mathcal{B}(\mathbf{R}(A^{1/2}))}^4}$. Hence the proof is complete by applying \ref{imporeq2009} together with Lemma \ref{norminraundemi}.
\end{proof}

\section{The $A$-seminorms-parallelism of $A$-bounded operators}\label{s3}
In this section, we study the parallelism of $A$-bounded operators with respect to $\|\cdot\|_A$ and $\omega_A(\cdot)$. We start by defining the first new kind of parallelism.
\begin{definition}
Let $T,S\in \mathcal{B}_{A^{1/2}}(\mathcal{H})$. We say that $T$ is $A$-seminorm-parallel to $S$, in short $T\parallel_A S$, if there exists $\lambda\in \mathbb{T}:=\{\alpha\in \mathbb{C}\,;\;|\alpha|=1 \}$ such that
$$\|T+\lambda S\|_A=\|T\|_A+\|S\|_A.$$
\end{definition}
\begin{example}
Let $T,S\in \mathcal{B}_{A^{1/2}}(\mathcal{H})$ be linearly dependent operators. Then there exists $\alpha\in\mathbb{C}\setminus \{0\}$ such that $S=\alpha T$.
By letting $\lambda=\frac{\overline{\alpha}}{|\alpha|}$, we see that
\begin{align*}
\|T + \lambda S\|_A= \|T + |\alpha| T\|_A = (1 + |\alpha|)\|T\|_A = \|T\|_A+\|S\|_A,
\end{align*}
So, we deduce that $T \parallel_A S$.
\end{example}

The following theorem gives a necessary and sufficient condition for $T\in \mathcal{B}_{A^{1/2}}(\mathcal{H})$ to be $A$-seminorm-parallel to $S\in \mathcal{B}_{A^{1/2}}(\mathcal{H})$. Our result generalizes Theorem \ref{paranorm}.
\begin{theorem}\label{main1}
Let $T,S\in \mathcal{B}_{A^{1/2}}(\mathcal{H})$. Then, the following assertions are equivalent:
\begin{itemize}
  \item [(1)] $T\parallel_AS$.
  \item [(2)] There exists a sequence $(x_n)_n\subset\mathcal{H}$ such that $\|x_n\|_A=1$,
    \begin{equation*}
    \lim_{n\to \infty}|\langle T x_n\mid Sx_n\rangle_A|=\|T\|_A\|S\|_A.
    \end{equation*}
\end{itemize}
\end{theorem}
\begin{proof}
(1)$\Longrightarrow$(2): Assume that $T\parallel_A S$. Then, $\|T+\lambda S\|_A=\|T\|_A+\|S\|_A$ for some $\lambda\in\mathbb{T}$. Moreover, by using \eqref{seminormold}, we infer that
$$\sup\{\|Tx+\lambda Sx\|_A\,;\; x\in\mathcal{H},\|x\|_A=1\}=\|T+\lambda S\|_A=\|T\|_A+\|S\|_A.$$
Hence, there exists a sequence $(x_n)_n\subset\mathcal{H}$ such that $\|x_n\|_A=1$ and
$$\displaystyle\lim_{n\rightarrow\infty} \|Tx_n+\lambda Sx_n\|_A=\|T\|_A+\|S\|_A.$$
 So, by using the Cauchy-Schwarz inequality and \eqref{ineqwich02} we see that
\begin{align*}
\|T\|_A^2+2\|T\|_A\,\|S\|_A+\|S\|_A^2
&=(\|T\|_A+\|S\|_A)^2\\
&=\lim_{n\rightarrow\infty} \|Tx_n+\lambda Sx_n\|_A^2\\
&=\lim_{n\rightarrow\infty} \left[\,\|Tx_n\|_A^2+\langle Tx_n\mid \lambda Sx_n\rangle_A+\langle\lambda Sx_n\mid Tx_n\rangle_A+\|Sx_n\|_A^2\,\right]\\
&\leq\|T\|_A^2+2\lim_{n\rightarrow\infty} |\langle Tx_n\mid Sx_n\rangle_A|+\|S\|_A^2\\
&\leq\|T\|_A^2+2\lim_{n\rightarrow\infty} \|Tx_n\|_A\|Sx_n\|_A+\|S\|_A^2\\
&\leq\|T\|_A^2+2\|T\|_A\,\|S\|_A+\|S\|_A^2.
\end{align*}
Therefore, we get
$$\lim_{n\rightarrow\infty} |\langle Tx_n\mid Sx_n\rangle_A|=\|T\|_A\,\|S\|_A.$$

(2)$\Longrightarrow$(1): Assume that there exists a sequence of $A$-unit vectors $(x_n)_n\subseteq\mathcal{H}$ such that
\begin{equation}\label{zamanicana}
\displaystyle{\lim_{n\rightarrow +\infty}}|{\langle Tx_n, Sx_n\rangle}_{A}| = {\|T\|}_{A}\,{\|S\|}_{A}.
\end{equation}
 Notice that if $AT=0$ or $AS=0$,
then $T\parallel_AS$. Assume that $AT\neq0$ and $AS\neq0$. Since the unit sphere of $\mathbb{C}$ is compact, then by taking a further subsequence we may assume that there is some $\lambda\in\mathbb{T}$ such that
$$\lim_{n\rightarrow\infty}\frac{\langle Tx_{n}\mid Sx_{n}\rangle_A}{|\langle Tx_{n}\mid Sx_{n}\rangle_A|}=\lambda.$$
This in turn implies, by using \eqref{zamanicana}, that
$$\lim_{n\rightarrow\infty}\langle Tx_{n}\mid Sx_{n}\rangle_A=\lim_{n\rightarrow\infty}\frac{\langle Tx_{n}\mid Sx_{n}\rangle_A}{|\langle Tx_{n}\mid Sx_{n}\rangle_A|}|\langle Tx_{n}\mid Sx_{n}\rangle_A|=\lambda\|T\|_A\|S\|_A.$$
Hence, $\displaystyle{\lim_{n\rightarrow +\infty}}{\langle Tx_n\mid \lambda Sx_n\rangle}_{A}= {\|T\|}_{A}\,{\|S\|}_{A}$. So, we deduce that
\begin{equation}\label{limreal}
\displaystyle{\lim_{n\rightarrow +\infty}}\mbox{Re}{\langle Tx_n\mid \lambda Sx_n\rangle}_{A}= {\|T\|}_{A}\,{\|S\|}_{A}.
\end{equation}
Moreover, by using the Cauchy-Schwarz inequality it follows from
$${\|T\|}_{A}\,{\|S\|}_{A}=\displaystyle{\lim_{n\rightarrow +\infty}}|{\langle Tx_n, Sx_n\rangle}_{A}|\leq \displaystyle{\lim_{n\rightarrow +\infty}}\|Tx_n\|_A\,\|S\|_A\leq {\|T\|}_{A}\,{\|S\|}_{A}$$
that $\displaystyle{\lim_{n\rightarrow +\infty}}\|Tx_n\|_A={\|T\|}_{A}$. In addition, by a similar argument, we get $\displaystyle{\lim_{n\rightarrow +\infty}}{\|Sx_n\|}_{A} = {\|S\|}_{A}$. So, by using \eqref{limreal}, we infer that
\begin{align*}
{\|T\|}_{A} +{\|S\|}_{A}
&\geq {\|T+\lambda S\|}_{A}\\
&\geq \left(\displaystyle{\lim_{n\rightarrow +\infty}}{\|(T+\lambda S)x_n\|}_{A}^2\right)^{1/2}\\
&\geq \left(\displaystyle{\lim_{n\rightarrow +\infty}}\left[{\|Tx_n\|}_{A}^2+2\mbox{Re}{\langle Tx_n\mid \lambda Sx_n\rangle}_{A}+{\|Sx_n\|}_{A}^2\right]\right)^{1/2}\\
& \leq \left({\|T\|}^2_{A} + 2{\|S\|}_{A}{\|T\|}_{A} + {\|S\|}^2_{A}\right)^{1/2}= {\|T\|}_{A} + {\|S\|}_{A}.
\end{align*}
Hence, we obtain ${\|T +\lambda S\|}_{A} = {\|T\|}_{A} + {\|S\|}_{A}$ and so $T\parallel_AS$.
\end{proof}

In order to prove one of our main results in this section, we need the following lemma which describes the connection between the $A$-seminorm parallelism of $A$-bounded operators and the norm-parallelism of operators in $\mathcal{B}(\mathbf{R}(A^{1/2}))$. Recall that an operator $T\in \mathcal{B}_{A^{1/2}}(\mathcal{H})$ if and only if there exists a unique $\widetilde{T}\in \mathcal{B}(\mathbf{R}(A^{1/2}))$ such that $Z_AT =\widetilde{T}Z_A$ (see Proposition \ref{prop_arias}).

\begin{lemma}\label{lempar01}
Let $T,S\in \mathcal{B}_{A^{1/2}}(\mathcal{H})$. Then, $T\parallel_AS$ if and only if $\widetilde{T}\parallel \widetilde{S}$.
\end{lemma}

\begin{proof}
We shall prove that $\|\widetilde{T+\lambda S}\|_{\mathcal{B}(\mathbf{R}(A^{1/2}))}=\|\widetilde{T}+\lambda \widetilde{S}\|_{\mathcal{B}(\mathbf{R}(A^{1/2}))}$ for any $\lambda\in \mathbb{C}$. Since, $T,S\in \mathcal{B}_{A^{1/2}}(\mathcal{H})$, then by Proposition \ref{prop_arias} there exists a unique $\widetilde{T},\widetilde{S}\in \mathcal{B}(\mathbf{R}(A^{1/2}))$ such that $Z_AT =\widetilde{T}Z_A$ and $Z_AS=\widetilde{S}Z_A$. So, $Z_A(T+\lambda S) =(\widetilde{T}+\lambda \widetilde{T})Z_A$ for any $\lambda\in \mathbb{C}$. On the other hand, $Z_A(T+\lambda S) =(\widetilde{T+\lambda S})Z_A$. Hence, $\widetilde{T+\lambda S}=\widetilde{T}+\lambda \widetilde{T}$ for all $\lambda\in \mathbb{C}$. This implies that $\|\widetilde{T+\lambda S}\|_{\mathcal{B}(\mathbf{R}(A^{1/2}))}=\|\widetilde{T}+\lambda \widetilde{S}\|_{\mathcal{B}(\mathbf{R}(A^{1/2}))}$ for any $\lambda\in \mathbb{C}$. On the other hand, by Lemma \ref{imporeq2009} we have $\|X\|_A=\|\widetilde{X}\|_{\mathcal{B}(\mathbf{R}(A^{1/2}))}$ for every $X\in \mathcal{B}_{A^{1/2}}(\mathcal{H})$. Hence, the proof is complete.
\end{proof}

We are now in a position to establish one of our main results in this section which gives a connection between the $A$-seminorm-parallelism to the identity operator and an equality condition for the $A$-Davis-Wielandt radius of $A$-bounded operators. Our result generalizes Theorems \ref{thm001}.
\begin{theorem}\label{equinew2}
Let $T\in \mathcal{B}_{A^{1/2}}(\mathcal{H})$. Then, the following assertions are equivalent:
\begin{itemize}
\item [(1)] $T\parallel_A I$.
\item [(2)] $d\omega_A(T)=\sqrt{\omega_A(T)^2+\|T\|_A^4}$.
\item [(3)] $T$ is $A$-normaloid.
\item [(4)] $d\omega_A(T)=\|T\|_A\sqrt{1+\|T\|_A^2}$.
\end{itemize}
\end{theorem}
\begin{proof}
$(1)\Leftrightarrow(2):$ Observe first that $\widetilde{I}=I_{\mathbf{R}(A^{1/2})}$. Moreover, by using Lemma \ref{norminraundemi} together with Theorem \ref{thm001}, we infer that
\begin{align*}
T\parallel_A I
&\Leftrightarrow \widetilde{T}\parallel I_{\mathbf{R}(A^{1/2})} \\
 &\Leftrightarrow d\omega(\widetilde{T})=\sqrt{\omega(\widetilde{T})^2+\|T\|_{\mathcal{B}(\mathbf{R}(A^{1/2}))}^4}\\
 &\Leftrightarrow d\omega_A(T)=\sqrt{\omega_A(T)^2+\|T\|_A^4}.
\end{align*}
$(2)\Leftrightarrow(3):$ Follows from Theorem \ref{equinew}.

 $(2)\Leftrightarrow(4):$ Assume that $d\omega_A(T)=\sqrt{\omega_A(T)^2+\|T\|_A^4}$. By using the equivalence $(2)\Leftrightarrow(3)$, we have $T$ is $A$-normaloid and then by Proposition \ref{charactenormailoid}, we deduce that $\omega_A(T)=\|T\|_A$. So, we get the assertion $(4)$ as required. Now, assume that $(4)$ holds. Thanks to \cite[Proposition 2.5.]{bakfeki01}, we have $\omega_A(T)\leq \|T\|_A$. This allows us to deduce, by using also the assertion $(5)$ of Proposition \ref{propositionsum}, that
 $$\|T\|_A\sqrt{1+\|T\|_A^2}=d\omega_A(T)\leq \sqrt{\omega_A(T)^2+\|T\|_A^4}\leq \|T\|_A\sqrt{1+\|T\|_A^2}.$$
 Hence, $d\omega_A(T)=\sqrt{\omega_A(T)^2+\|T\|_A^4}$.
\end{proof}

In the following corollary, we investigate the case when an operator $T\in \mathcal{B}_{A}(\mathcal{H})$ is parallel to the identity operator.
\begin{corollary}
Let $T\in \mathcal{B}_{A}(\mathcal{H})$. Then, the following assertions are equivalent:
\begin{itemize}
  \item [(1)] $T\parallel_A I$.
  \item [(2)] $ \omega_A(T)^2I\geq_A T^\sharp T$ that is $\langle A(\omega_A(T)^2I-T^\sharp T)x\mid x \rangle\geq 0,\;\forall\,x\in \mathcal{H}$.
\end{itemize}
\end{corollary}
\begin{proof}
One can observe that
\begin{align*}
\omega_A(T)=\|T\|_A
&\Leftrightarrow  \|Tx\|_A\leq \omega_A(T)\|x\|_A,\;\forall\,x\in \mathcal{H}\\
 &\Leftrightarrow  \|Tx\|_A^2\leq \omega_A(T)^2\|x\|_A^2,\;\forall\,x\in \mathcal{H}\\
  &\Leftrightarrow  \langle Tx\mid Tx \rangle_A\leq \langle \omega_A(T)^2x\mid x \rangle_A,\;\forall\,x\in \mathcal{H}\\
    &\Leftrightarrow  \langle A(T^\sharp T-\omega_A(T)^2I)x\mid x \rangle\leq 0,\;\forall\,x\in \mathcal{H}.
\end{align*}
On the other hand, by Proposition \ref{charactenormailoid}, $T$ is $A$-normaloid if and only if $\omega_A(T)=\|T\|_A$. Hence the proof is complete by using Theorem \ref{equinew2}.
\end{proof}

Now, we aim to study the case when $\mathcal{H}$ be a finite dimensional Hilbert space. We first state the following theorem.
\begin{theorem}\label{finite01}
Let $\mathcal{H}$ be a finite dimensional Hilbert space and $T,S\in\mathcal{B}_{A^{1/2}}(\mathcal{H})$. Then the following conditions are equivalent:
\begin{itemize}
\item[(1)] $T\parallel_AS$.
\item[(2)] There exists an $A$-unit vector $x\in\mathcal{H}$ such that
\begin{align*}
|\langle T x\mid Sx\rangle_A|=\|T\|_A\|S\|_A.
\end{align*}
\end{itemize}
\end{theorem}
\begin{proof}
(2)$\Longrightarrow$(1) Follows immediately by applying Theorem \ref{main1}.

(1)$\Longrightarrow$(2) Assume that $T\parallel_AS$. Then, by Lemma \ref{lempar01} we have $\widetilde{T}\parallel \widetilde{S}$. So, by Theorem \ref{main1} there exists a sequence $(y_n)_n\subset\mathbf{R}(A^{1/2})$ such that $\|y_n\|_{\mathbf{R}(A^{1/2})}=1$,
    \begin{equation}\label{marj}
    \lim_{n\to \infty}|(\widetilde{T}y_n, \widetilde{S}y_n)|= \|\widetilde{T}\|_{\mathcal{B}(\mathbf{R}(A^{1/2}))}\,\|\widetilde{S}\|_{\mathcal{B}(\mathbf{R}(A^{1/2}))}.
    \end{equation}
On the other hand, since $\mathcal{H}$ be a finite dimensional Hilbert space, then the Hilbert space $\mathbf{R}(A^{1/2}))$ is also finite dimensional. So, the closed unit sphere of $\mathbf{R}(A^{1/2}))$ is compact. Therefore, $(y_n)_n$ has a subsequence $(y_{n_k})_k$ which converges to some $y\in \mathbf{R}(A^{1/2}))$ with $\|y\|_{\mathbf{R}(A^{1/2})}=1$. Hence, by using \eqref{marj}, we get
$$|(\widetilde{T}y, \widetilde{S}y)|=\lim_{k\to \infty}|(\widetilde{T}y_{n_k}, \widetilde{S}y_{n_k})|= \|\widetilde{T}\|_{\mathcal{B}(\mathbf{R}(A^{1/2}))}\,\|\widetilde{S}\|_{\mathcal{B}(\mathbf{R}(A^{1/2}))}.$$
Now, since $\mathcal{H}$ be a finite dimensional, then it can be seen that $\mathcal{R}(A)=\mathcal{R}(A^{1/2})$. Thus, there exists $x\in\mathcal{H}$ such that $y=Ax$. This implies that $\|x\|_A=\|Ax\|_{\mathbf{R}(A^{1/2})}=1$ and
$$|(\widetilde{T}Ax, \widetilde{S}Ax)|=|(ATx, ASx)|=\|\widetilde{T}\|_{\mathcal{B}(\mathbf{R}(A^{1/2}))}\,\|\widetilde{S}\|_{\mathcal{B}(\mathbf{R}(A^{1/2}))}.$$
Therefore, we get the desired result by applying Lemma \ref{imporeq2009} together with \eqref{usefuleq001}.
\end{proof}
Now, we are able to prove the following theorem.
\begin{theorem}\label{Cr.3.10}
Let $\mathcal{H}$ be a finite dimensional Hilbert space and $T\in \mathcal{B}_{A^{1/2}}(\mathcal{H})$. The following statements are equivalent:
\begin{itemize}
\item[(i)] $d\omega_A(T) = \sqrt{\omega_A^2(T) + \|T\|_A^4}$.
\item[(ii)] There exists an $A$-unit vector $x\in \mathcal{H}$ such that $|\langle Tx\mid x\rangle_A| = \|T\|_A$.
\end{itemize}
\end{theorem}
\begin{proof}
This follows immediately from Theorems \ref{finite01} and \ref{equinew2}.
\end{proof}

Now, we turn to introduce and characterize investigate the parallelism of $A$-bounded operators with respect to $\omega_A(\cdot)$.
\begin{definition}
Let $T,S\in\mathcal{B}_{A^{1/2}}(\mathcal{H})$. The operator $T$ is said to be $A$-numerical radius parallel to $S$ and we denote $T \parallel_{\omega_A} S$, if
\begin{align*}
\omega_A(T + \lambda S) = \omega_A(T)+\omega_A(S), \quad \text{for some}\,\, \lambda\in\mathbb{T}.
\end{align*}
\end{definition}

Now, we aim to extend Theorem \ref{th.1} to the class of $A$-bounded operators. To do this, we need the following lemma. Its proof follows immediately by following the proof of Lemma \ref{lempar01} and using Lemma \ref{imporeq2009}.
\begin{lemma}\label{lempar02}
Let $T,S\in \mathcal{B}_{A^{1/2}}(\mathcal{H})$. Then, the following assertions are equivalent:
\begin{itemize}
  \item [(1)] $T\parallel_{\omega_A}S$.
  \item [(2)] $\widetilde{T}\parallel_{\omega} \widetilde{S}$.
\end{itemize}
\end{lemma}
Now, we are in a position to prove one of our main results in this section which generalizes Theorem \ref{th.1}.
\begin{theorem}\label{main2}
Let $T,S\in \mathcal{B}_{A^{1/2}}(\mathcal{H})$. Then the following conditions are equivalent:
\begin{itemize}
\item[(1)] $T \parallel_{\omega_A} S$.
\item[(2)] There exists a sequence of $A$-unit vectors $\{x_n\}$ in $\mathcal{H}$ such that
\begin{equation}\label{importanteqf}
\lim_{n\rightarrow\infty} \big|\langle Tx_n\mid x_n\rangle_A\langle Sx_n\mid x_n\rangle_A\big| = \omega_A(T)\omega_A(S).
\end{equation}
\end{itemize}
\end{theorem}
\begin{proof}
(2)$\Longrightarrow$(1): Assume that there exists a sequence $(x_n)_n\subset\mathcal{H}$ such that $\|x_n\|_A=1$ and
    \begin{equation}\label{num2019}
\lim_{n\rightarrow\infty} \big|\langle Tx_n\mid x_n\rangle_A\langle Sx_n\mid x_n\rangle_A\big| = \omega_A(T)\omega_A(S).
    \end{equation}
By $y_n=Ax_n$ and using \eqref{usefuleq01}, we infer that $\|y_n\|_{\mathbf{R}(A^{1/2})}=\|Ax_n\|_{\mathbf{R}(A^{1/2})}=\|x_n\|_A=1$. Moreover, in view of \eqref{usefuleq001}, we have
$$\langle Tx_n\mid x_n\rangle_A\langle Sx_n\mid x_n\rangle_A=(\widetilde{T}y_n,y_n)(\widetilde{S}y_n,y_n).$$
So, by using \eqref{num2019} together with Lemma \ref{imporeq2009} we get
    \begin{equation*}
    \lim_{n\to \infty}|(\widetilde{T}y_n,y_n)(\widetilde{S}y_n,y_n)|=\omega(\widetilde{T})\omega(\widetilde{S}).
    \end{equation*}
This yields, by Theorem \ref{th.1}, that $\widetilde{T}\parallel_{\omega} \widetilde{S}$ which in turn implies, in view of Lemma \ref{lempar02}, that $T\parallel_{\omega_A}S$.

(1)$\Longrightarrow$(2): Let $T \parallel_{\omega_A} S$. Then there exists $\lambda\in\mathbb{T}$ such that
$\omega_A(T+\lambda S) = \omega_A(T)+\omega_A(S)$. On the other hand, by the definition of the $A$-numerical radius we have
\begin{align*}
\omega_A(T+\lambda S)=\sup\Big\{\big|\langle(T+\lambda S)x\mid x\rangle_A\big|: x\in \mathcal{H},\|x\|_A=1\Big\}.
\end{align*}
So, there exists a sequence $(x_n)_n\subset\mathcal{H}$ such that $\|x_n\|_A=1$ and
\begin{align*}
\lim_{n\rightarrow\infty} \big|\langle(T+\lambda S)x_n\mid x_n\rangle_A\big| = \omega_A(T+\lambda S).
\end{align*}
Moreover, since $|\langle Tx_n,x_n\rangle_A|\leq \omega_A(T)$ and $|\langle Sx_n,x_n\rangle_A|\leq \omega_A(S)$ for all $n$, then it follows that
\begin{align*}
\big(\omega_A(T)+\omega_A(S)\big)^2
& = \omega_A^2(T+\lambda S)\\
& = \lim_{n\rightarrow\infty} \big|\langle(T+\lambda S)x_n\mid x_n\rangle_A\big|^2\\
& = \lim_{n\rightarrow\infty}\Big(|\langle Tx_n\mid x_n\rangle_A|^2+ 2\mbox{Re}\big(\lambda\langle Tx_n\mid x_n\rangle_A \langle Sx_n\mid x_n\rangle_A\big)
+ |\langle S x_n\mid x_n\rangle_A|^2\Big)\\
&\leq\lim_{n\rightarrow\infty}\Big(|\langle Tx_n\mid x_n\rangle_A|^2+2\big|\langle Tx_n\mid x_n\rangle_A\langle Sx_n\mid x_n\rangle_A\big|+|\langle Sx_n\mid x_n\rangle_A|^2\Big)\\
&\leq \omega_A^2(T)+2\lim_{n\rightarrow\infty}\big|\langle Tx_n\mid x_n\rangle_A\langle Sx_n\mid x_n\rangle_A\big|+\omega_A^2(S)\\
&\leq\omega_A^2(T)+2\omega_A(T)\omega_A(S)+\omega_A^2(S)= \big(\omega_A(T)+\omega_A(S)\big)^2.
\end{align*}
Hence, we deduce that
\begin{align*}
\lim_{n\rightarrow\infty}\big|\langle Tx_n\mid x_n\rangle_A\langle Sx_n\mid x_n\rangle_A\big| = \omega_A(T)\omega_A(S).
\end{align*}
Therefore, the proof is complete.
\end{proof}
Before we move on, let us emphasize the following remark.
\begin{remark}
If there exists a sequence of $A$-unit vectors $\{x_n\}$ in $\mathcal{H}$ satisfying \eqref{importanteqf}, then it also satisfies $\displaystyle{\lim_{n\rightarrow\infty}}|\langle Tx_n\mid x_n\rangle_A| = \omega_A(T)$
and $\displaystyle{\lim_{n\rightarrow\infty}}|\langle Sx_n\mid x_n\rangle_A| = \omega_A(S)$. Indeed, if $AT=0$ or $AS=0$ then the result follows trivially. Assume that $AT\neq0$ and $AS\neq0$, then $\omega_A(T)\neq 0$ and $\omega_A(S)\neq 0$. So, it follows from
\begin{align*}
\omega_A(T)\omega_A(S)=\lim_{n\rightarrow\infty}\big|\langle Tx_n\mid x_n\rangle_A\langle Sx_n\mid x_n\rangle_A\big|
\leq \lim_{n\rightarrow\infty}|\langle Tx_n\mid x_n\rangle_A|\omega_A(S)\leq \omega_A(T)\omega_A(S)
\end{align*}
that $\displaystyle{\lim_{n\rightarrow\infty}}|\langle Tx_n\mid x_n\rangle_A| = \omega_A(T)$ and by using a similar argument,
$\displaystyle{\lim_{n\rightarrow\infty}}|\langle Sx_n\mid x_n\rangle_A| = \omega_A(S)$.
\end{remark}
 The following corollary is an immediate consequence of Theorem \ref{main2}. Its proof is similar to that of Theorem \ref{finite01} and so we omit it.
\begin{corollary}
Let $\mathcal{H}$ be a finite dimensional Hilbert space and $T,S\in \mathcal{B}_{A^{1/2}}(\mathcal{H})$. Then the following conditions are equivalent:
\begin{itemize}
\item[(1)] $T \parallel_{\omega_A} S$.
\item[(2)] There exists an $A$-unit vector $x\in\mathcal{H}$ such that
\begin{align*}
|\langle Tx\mid x\rangle_A\langle Sx\mid x\rangle_A| = \omega_A(T)\omega_A(S).
\end{align*}
\end{itemize}
\end{corollary}

Recall from \cite{kais01} that an operator $T\in\mathcal{B}_A(\mathcal{H})$ is said to be $A$-hyponormal if
$$T^{\sharp}T\geq_ATT^{\sharp}\Leftrightarrow \langle A(T^{\sharp}T-TT^{\sharp})x\mid x\rangle\geq 0,\quad \forall\,x\in \mathcal{H}.$$
Now, we state the following proposition.
\begin{proposition}
Let $T,S\in \mathcal{B}_A(\mathcal{H})$ be two $A$-hyponormal operators such that $T \parallel_{\omega_A} S$. Then, $T \parallel_A S$.
\end{proposition}
\begin{proof}
Notice first that since $T$ and $S$ are $A$-hyponormal operators, then by \cite{kais01} we have $\omega_A(T)=\|T\|_A$ and $\omega_A(S)=\|S\|_A$. On the other hand, since $T\parallel_{\omega_A} S$, then $\omega_A(T + \lambda S) = \omega_A(T) + \omega_A(S)$ for some $\lambda\in\mathbb{T}$. So, by using the fact that $\omega_A(X)\leq \|X\|_A$ for every $A$-bounded operator $X$, we get
\begin{align*}
\omega_A(T) + \omega_A(S)
&= \omega_A(T + \lambda S)\\
&\leq \|T + \lambda S\|_A\leq \|T\|_A + \|S\|_A = \omega_A(T) + \omega_A(S).
\end{align*}
Hence $\|T + \lambda S\|_A = \|T\|_A + \|S\|_A$ and therefore $T\parallel_A S$.
\end{proof}

As an immediate consequence of the above proposition we have the following result.
\begin{corollary}
Let $T,S\in \mathcal{B}_A(\mathcal{H})$ be two $A$-normal operators $($i.e. $T^{\sharp}T=TT^{\sharp}$ and $S^{\sharp}S=SS^{\sharp}$$)$ such that $T \parallel_{\omega_A} S$. Then, $T \parallel_A S$.
\end{corollary}

\end{document}


\documentclass[12pt,reqno]{amsart}
\usepackage{etoolbox}

   \makeatletter

 \patchcmd{\@setaddresses}{\scshape\ignorespaces}{\ignorespaces}{}{} 

\appto\maketitle{%
\let\@makefnmark\relax  \let\@thefnmark\relax
\ifx\@empty\addresses\else\@footnotetext{%
  \vskip-\bigskipamount\@setaddresses}
  }
\def\enddoc@text{}
\makeatother

\makeatletter
\patchcmd\maketitle
  {\uppercasenonmath\shorttitle}
  {}
  {}{}
\patchcmd\maketitle
  {\@nx\MakeUppercase{\the\toks@}}
  {\the\toks@}
  {}
  {}{}
\patchcmd\@settitle{\uppercasenonmath\@title}{\Large}{}{}
\patchcmd\@setauthors
  {\MakeUppercase{\authors}}
  {\authors}
  {}{}
\makeatother
\usepackage{amsmath,amssymb,amsthm}
\usepackage{color}
\usepackage{url}
\usepackage{tikz-cd}
\usepackage[utf8]{inputenc}
\usepackage[T1]{fontenc}
\usepackage{geometry}
\geometry{left=2cm,right=2cm,top=2cm,bottom=2cm}
\newtheorem{theorem}{Theorem}[section]
\newtheorem{definition}{Definition}[section]
\newtheorem{definitions}{Definitions}[section]
\newtheorem{notation}{Notation}[section]
\newtheorem{corollary}{Corollary}[section]
\newtheorem{proposition}{Proposition}[section]
\newtheorem{lemma}{Lemma}[section]
\newtheorem{remark}{Remark}[section]
\newtheorem{example}{Example}[section]
\numberwithin{equation}{section}
\usepackage[colorlinks=true]{hyperref}  
\hypersetup{urlcolor=blue, citecolor=red, linkcolor=blue}

\usepackage[capitalise,noabbrev,nameinlink]{cleveref}      
\begin{document}
\address{$^{[1]}$ University of Sfax, Tunisia.}
\email{\url{kais.feki@hotmail.com}}

\address{$^{[2]}$ Mathematics Department, College of Science, Jouf University, P.O. Box 2014, Sakaka, Saudi Arabia.}
\email{\url{sidahmed@ju.edu.sa}}



\date{\today}
\author[Kais Feki and Sid Ahmed Ould Ahmed Mahmoud] {\Large{Kais Feki}$^{1}$ and \Large{Sid Ahmed Ould Ahmed Mahmoud}$^{2}$ }
\title[Davis-Wielandt shells of semi-Hilbertian space operators and its applications]{Davis-Wielandt shells of semi-Hilbertian space operators and its applications}

\maketitle